\documentclass[11pt, hyperref]{article}
\usepackage[a4paper]{geometry}
\usepackage[colorlinks,linkcolor=black,anchorcolor=black,citecolor=black]{hyperref}
\usepackage{authblk}
\usepackage{graphicx}
\usepackage{subfigure}
\usepackage{epstopdf}
\usepackage{amsmath}
\usepackage{amsthm}
\usepackage{amsfonts}
\usepackage{amssymb}
\usepackage{pifont}
\newtheorem{definition}{Definition}[section]
\newtheorem{lemma}{Lemma}[section]

\newtheorem{theorem}{Theorem}[section]

\newtheorem{remark}{Remark}[section]
\numberwithin{equation}{section}
\allowdisplaybreaks[4]

\usepackage{ulem}

\title{Effective Numerical Simulations of Synchronous Generator System}
\author[1, 2]{Jiawei Zhang}
\author[1, 2]{Aiqing Zhu}
\author[3]{Feng Ji}
\author[3]{Chang Lin}
\author[1, 2]{Yifa Tang}
\affil[1]{LSEC, ICMSEC, Academy of Mathematics and Systems Science, Chinese Academy of Sciences, Beijing 100190, China.}
\affil[2]{School of Mathematical Sciences, University of Chinese Academy of Sciences, Beijing 100049, China.}
\affil[3]{State Key Laboratory of Advanced Transmission Technology (State Grid Smart Grid Research Institute Co., Ltd), Beijing 102200, China.}

\date{}

\begin{document}
\maketitle

\begin{abstract}
Synchronous generator system is a complicated dynamical system for energy transmission, which plays an important role in modern industrial production. In this article, we propose some predictor-corrector methods and structure-preserving methods for a generator system based on the first benchmark model of subsynchronous resonance, among which the structure-preserving methods preserve a Dirac structure associated with the so-called port-Hamiltonian descriptor systems. To illustrate this, the simplified generator system in the form of index-1 differential-algebraic equations has been derived. Our analyses provide the global error estimates for a special class of structure-preserving methods called Gauss methods, which guarantee their superior performance over the PSCAD/EMTDC and the predictor-corrector methods in terms of computational stability. Numerical simulations are implemented to verify the effectiveness and advantages of our methods. \\     
{\bf{Keywords}}
\ Synchronous generator system, Predictor-corrector method, Structure-preserving method, Port-Hamiltonian descriptor system, Differential-algebraic equations
\end{abstract}

\section{Introduction}
The power system is an energy production and consumption system composed of power plants, transmission and distribution lines, power supply and distribution stations, and electricity consumption. Over the past hundred years, the power system has gained substantial progress, among which the successful development of three-phase alternating current (AC) synchronous generator has become a milestone for the great advancement of the power system. In recent decades, the power system has developed into a complicated dynamical system for energy transmission, with the electromechanical transient model playing a vital role in theoretical analysis and practical application (see \cite{PSSC}). Actually, since the introduction of synchronous generator models into the electromagnetic transient programs in the 1970s (see \cite{SGM}), the electromagnetic transient simulations, which were primarily used to analyze the electromagnetic transient processes in power networks, have been gradually applied to the study of the electromechanical transient processes. Classical Electro-Magnetic Transient Program (EMTP) softwares divide the entire system into three modules: circuit, generator and mechanical shaft. During the numerical simulation, each module operates independently and exchanges data with each other, which inevitably leads to a delay of one time step and a decline of numerical accuracy (see \cite{PSETS}). The main reason for this phenomenon is that the electric network equations for the electromechanical transient model are presented in the form of algebraic equations, which leads to the node voltages not being written as the state quantities of the differential equations for this model. In this way, it is difficult to implement the existing numerical methods solving ordinary differential equations to the electromechanical transient model of power system. This shortage motivates us to construct a more reasonable model that allows the node voltages to be part of the state. To this end, Ji et al. proposed a novel modeling approach in \cite{DMMS} for the AC synchronous generator system based on the first benchmark model of subsynchronous resonance, by means of the Euler-Lagrange equation with the node flux linkages and mechanical angular displacements as the generalized coordinates. For this new model, the electromagnetic transient process of the synchronous generator system can be numerically simulated at the microsecond level.

The predictor-corrector methods are one of the effective numerical integrators for the newly constructed synchronous generator system. On the whole, they can be considered as a generalization of the classical Adams-Bashforth-Moulton method, which is well known for the numerical solution of first-order differential equations (see \cite[Section III.1]{SODE1}). The main idea of the predictor-corrector technique is to derive a method with better convergence property through the appropriate combination of an explicit method and an implicit method. For instance, consider the following well-known initial value problem for first-order differential equation
\begin{align*}
\dot{v}=f(t,v(t)),\ v(0)=v_0,    
\end{align*}
where $f$ represents an arbitrary differentiable function with sufficient smoothness. The most common predictor-corrector method is the improved Euler method, who takes the forward Euler method as a predictor equation to obtain the preliminary approximation $v_{n+1}^{[0]}=v_n+hf(t_n,v_n)$ of the exact solution $v(t_{n+1})$, here $h$ is the time step, $t_n=nh$ and $v_n$ represents the numerical solution at the moment $t_n$. Then the final calculated solution $v_{n+1}$ will be given by a corrector equation based on the trapezoidal rule, which reads         
\begin{align*}
v_{n+1}=v_n+\frac{h}{2}\left(f\left(t_{n+1},v_{n+1}^{[0]}\right)+f(t_n,v_n)\right).    
\end{align*}
Note that the improved Euler method proposed above is an explicit method of order $2$, which indicates that the combination of the forward Euler method and trapezoidal rule gains higher accuracy than Euler method without too much increase in the computational complexity. Therefore, in view of the slow variation of the mechanical angular velocities within a time step of the electromagnetic transient calculation, the predictor-corrector methods for synchronous generator system will be built in this article following the idea of the improved Euler method, which will show excellent accuracy in numerical simulations.  

Apart from the predictor-corrector methods, there also exist a series of structure-preserving methods that are suitable for numerical simulation of the synchronous generator system. Generally speaking, structure-preserving algorithms are numerical methods constructed by preserving the inherent structure and characteristic properties of a system, which have the feature of long-term computational stability. Over the past few decades, structure-preserving methods have been widely applied in various areas such as molecular dynamics, quantum physics and astrodynamics, among which the symplectic method for Hamiltonian systems is a typical representative (see \cite{S1,S2,S3}). Compared to conventional integrators such as explicit Runge-Kutta methods, structure-preserving methods, in particular the symplectic methods, demonstrate superior long-term behavior including the slower error growth and approximate preservation of the energy. In consideration of the unique advantages of structure-preserving methods, many related researches have been taken in recent years. For instance, explicit symplectic or K-symplectic algorithms have been developed for charged particle dynamics (see \cite{CPD1,CPD3,CPD4,CPD5}). In addition, adaptive symplectic methods for simulating charged particle dynamics are also studied in \cite{CPD2}. On the other hand, structure-preserving methods still maintain their advantageous performance in the computational simulation for gyrocenter dynamics, which can be seen in \cite{GCD1,GCD2}. In terms of the nonlinear Schrodinger equation, Zhu et al. propose a symplectic simulation method for the motion of dark solitons in \cite{NLS1}, and Zhang et al. put forward revertible and symplectic methods for the Ablowitz-Ladik discrete nonlinear Schrodinger equation in \cite{NLS2}. As for the Vlasov-Maxwell system, there have been canonical or non-canonical symplectic particle-in-cell algorithms to simulate it (see \cite{PIC1,PIC2}), among which the method presented in  \cite{PIC1} is applicable to long-term large-scale simulations. Moreover, Tu et al. present high order symplectic integrators given by generating functions for many-body problem in \cite{NHO}, and Zhu et al. employ splitting technique to derive K-symplectic methods for non-canonical separable Hamiltonian systems in \cite{SKS}. Just recently, Zhu et al. proposed explicit K-symplectic methods for some nonseparable non-canonical Hamiltonian systems in \cite{EKS}. 

Inspired by the outstanding computational stability of structure-preserving methods demonstrated in the above mentioned works, we will propose structure-preserving methods for the  synchronous generator system. These methods preserve a Dirac structure associated with port-Hamiltonian descriptor systems, and more details about this kind of systems and Dirac structure can be found in \cite{pHDAE,PHST}. In this article, we will also perform numerical simulations of a structure-preserving method, then compare it with predictor-corrector methods and PSCAD/EMTDC (a widely-used professional software for electromagnetic transient simulation). Numerical results show that both predictor-corrector methods and the structure-preserving method possess significantly better performance over PSCAD/EMTDC, and structure-preserving method exhibits the best performance in terms of long-term computational stability.

This article is organized as follows. In Section \ref{section 2}, we briefly introduce the synchronous generator system derived by the novel modeling approach, where the main parameters of this system will be given. Section \ref{section 3} concentrates on the construction of predictor-corrector methods following the idea of the improved Euler method, then structure-preserving methods will be presented in Section \ref{section 4} with their Dirac-structure preservation. In Section \ref{section 5}, numerical simulations of the methods proposed in previous sections are carried out to verify their numerical behaviours. Afterwards, a brief summary will be made in Section \ref{section 6}. Finally, Appendix \ref{appendix} will complete the proof of the global error estimates for a special class of structure-preserving methods.

\section{Synchronous generator system\label{section 2}}
This article focuses on a synchronous generator system based on the first benchmark model of subsynchronous resonance (see \cite{FBMSR}), whose concrete details are shown in the following Figure \ref{Fig:ACSGS}. In brief, line resistance $R=0.5\ \mathrm{m}\Omega$ and inductance $L=0.6182\ \mathrm{mH}$ are considered as part of the generator system in Figure \ref{Fig:ACSGS}(a), series-connected with an ideal AC voltage source whose amplitude is $U_s=26\ \mathrm{kV}$. Moreover, the synchronous generator in Figure \ref{Fig:ACSGS}(b) has three output ports $a,b,c$, which can be simplified to two directions $\alpha,\beta$ by Clarke transformation; $R_f, R_q$ are the resistors of the excitation winding and the damper winding, respectively. As for the mechanical shaft presented in Figure \ref{Fig:ACSGS}(c), more details could be seen in \cite{FBMSR}.  
\begin{figure}[!ht]
   \centering
   \subfigure[Circuit part.]{\begin{minipage}[t]{0.3\linewidth}
   \centering
   \includegraphics[width=4cm]{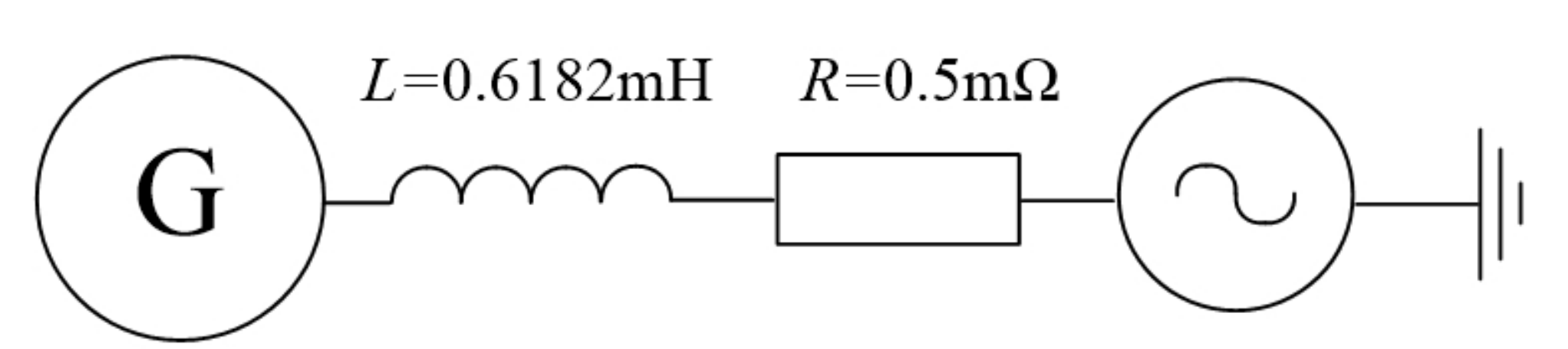}
   \end{minipage}}
   \subfigure[Synchronous generator.]{\begin{minipage}[t]{0.3\linewidth}
   \centering
   \includegraphics[width=4cm]{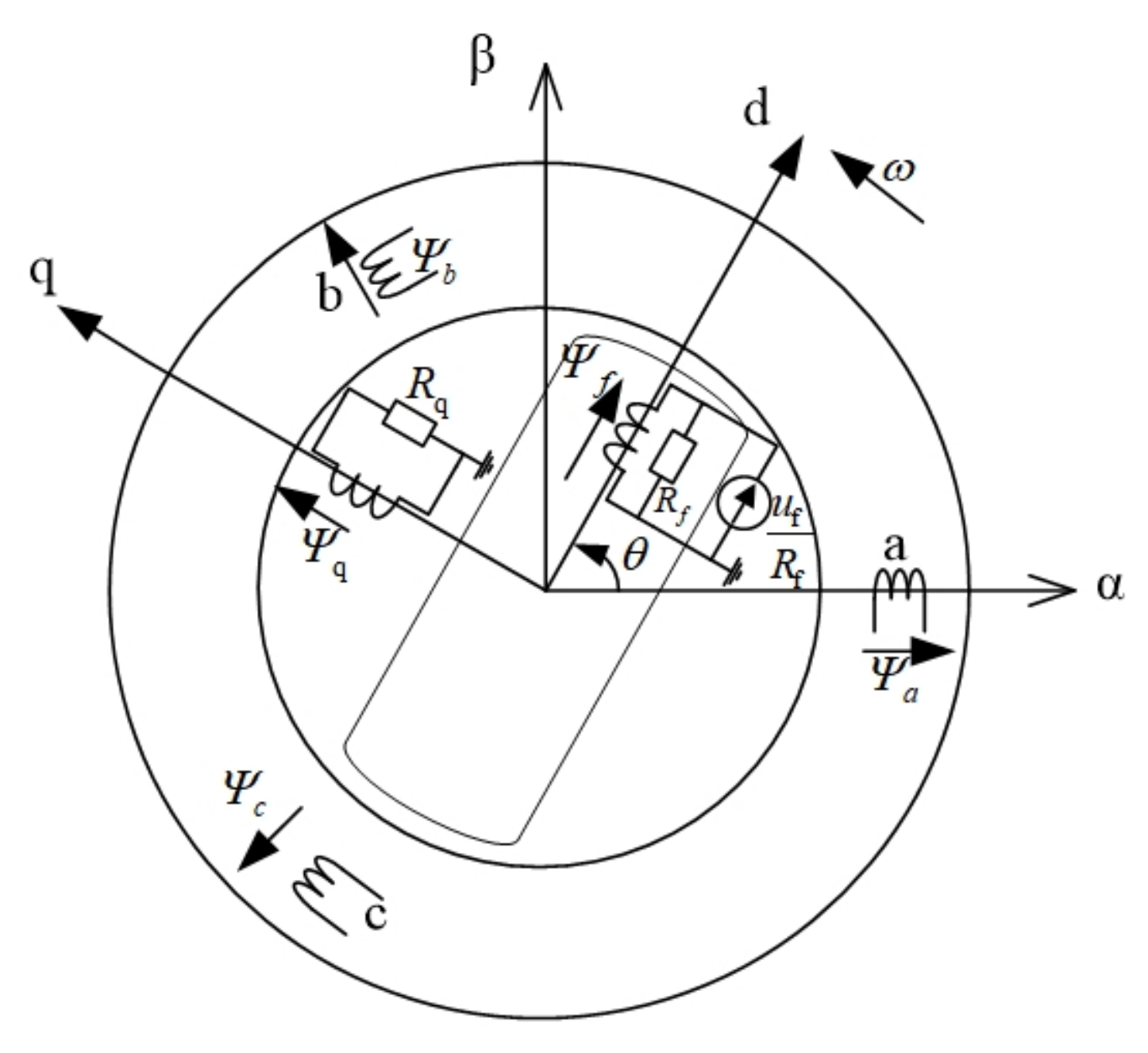}
   \end{minipage}}
    
   \subfigure[Mechanical shaft.]{\begin{minipage}[t]{0.5\linewidth}
   \centering
   \includegraphics[width=6cm]{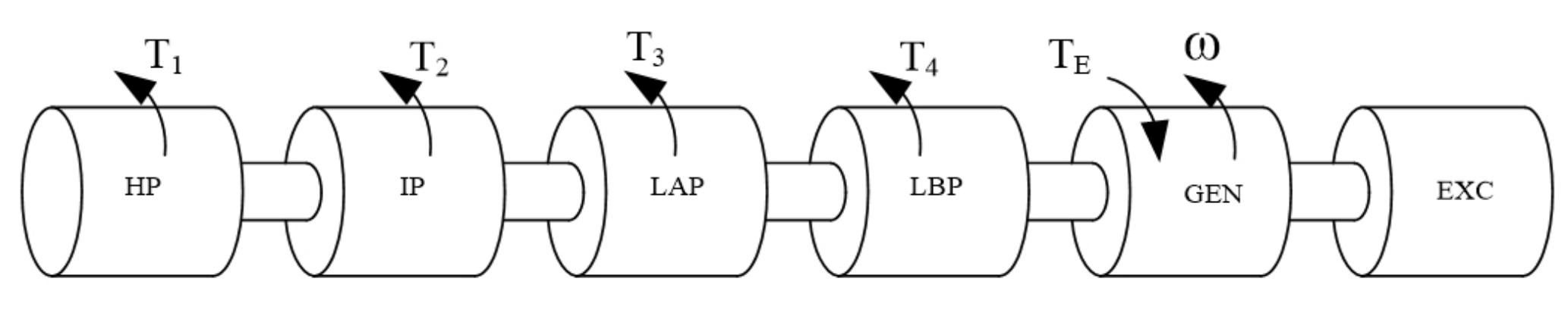}
   \end{minipage}}
   
   \caption{Synchronous generator system.}
   \label{Fig:ACSGS}
\end{figure}

\subsection{State quantities of the generator system\label{section 2.1}}
According to the modelling method presented in \cite{DMMS}, the dynamical equations of the generator system shown in Figure \ref{Fig:ACSGS} can be obtained based on the appropriate state quantities. First of all, the ideal AC voltage source in Figure \ref{Fig:ACSGS}(a) should be substituted by the equivalent Norton current source, which leads to the following circuit structure in Figure \ref{Fig:Norton}. Number the three circuit nodes by $0,1,2$ in Figure \ref{Fig:Norton}, among which the node $0$ locates at the grounding point. Notably, the voltage and the flux linkage of node $0$ vanish in this system, so they will not appear in the dynamical equations.    
\begin{figure}[!ht]
    \centering
    \includegraphics[width=5cm]{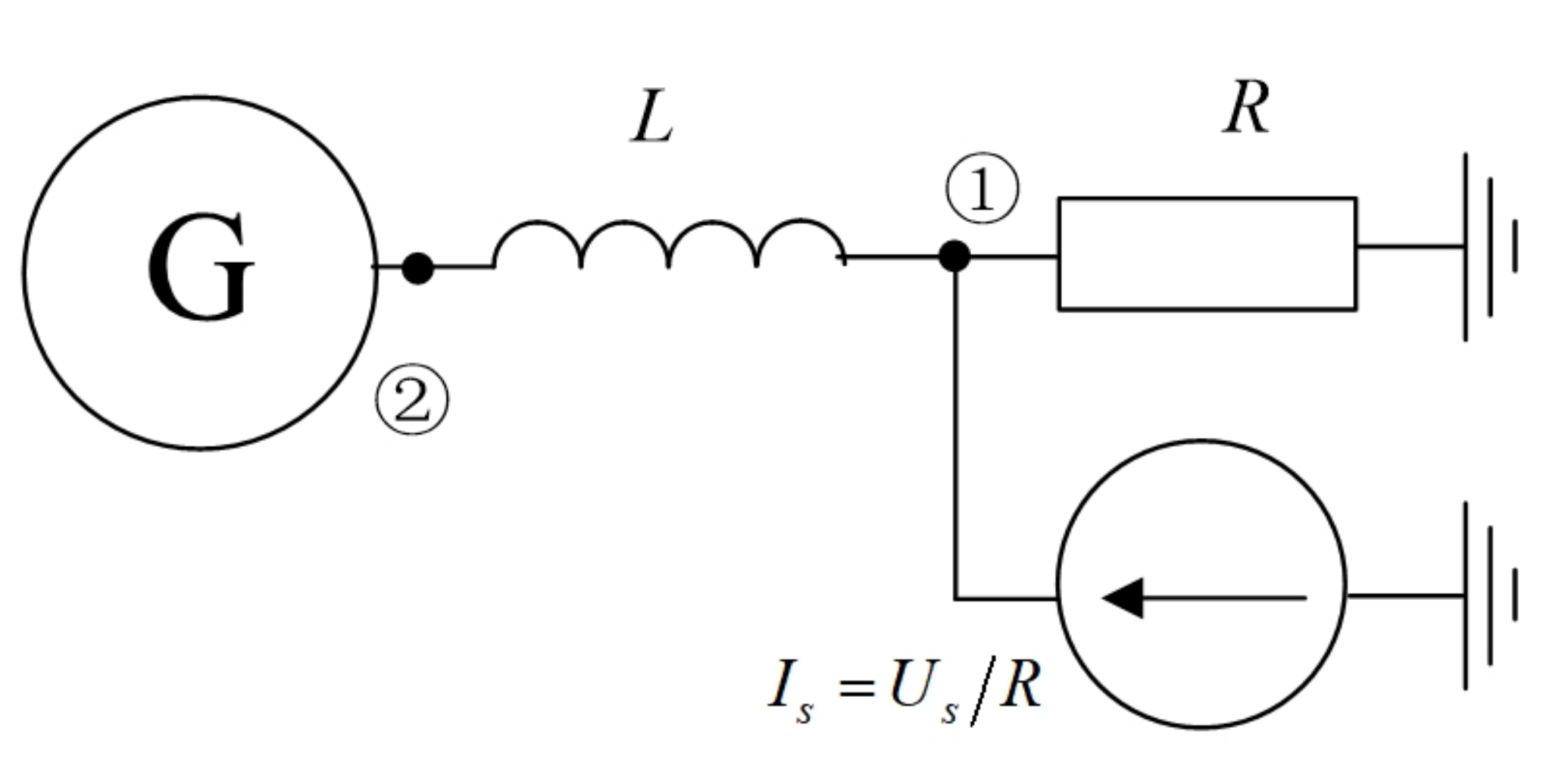}
    \caption{Equivalent model of circuit part.}
    \label{Fig:Norton}
\end{figure}

Now the dynamical equations of the generator system can be given through the flux linkages and the angular displacements who describe the state of this system. In general, there are six flux linkages associated with the two nodes $1,2$, the excitation winding and the damper winding, which can be written in vector form   
\begin{align}\label{Eq. flux}
\boldsymbol{\Psi}=(\Psi_{1\alpha},\Psi_{1\beta},\Psi_{2\alpha},\Psi_{2\beta},\Psi_f,\Psi_q)^{\top}.
\end{align}
 Here $\Psi_{1\alpha},\Psi_{1\beta}$ are two components of the flux linkage at node $1$, corresponding to the $\alpha,\beta$ directions respectively. In addition, $\Psi_{2\alpha},\Psi_{2\beta}$ have the analogous meaning to $\Psi_{1\alpha},\Psi_{1\beta}$, while $\Psi_f,\Psi_q$ are the flux linkages of the excitation winding and the damper winding. It is worth mentioning that $\dot{\boldsymbol{\Psi}}=(U_{1\alpha},U_{1\beta},U_{2\alpha},U_{2\beta},U_f,U_q)^{\top}$ is the vector composed of the voltages that reflect the system status.

As for the angular displacements, it is natural that      
\begin{align}\label{Eq. angle}
\boldsymbol{\theta}=(\theta_1,\theta_2,\theta_3,\theta_4,\theta_5,\theta_6)^{\top}   
\end{align}
plays an important role in depicting the generator system, where $\theta_i\ (i=1,\cdots,6)$ represent the angular displacements of the six mass blocks in Figure \ref{Fig:ACSGS}(c). Moreover, $\dot{\boldsymbol{\theta}}=(\omega_1,\omega_2,\omega_3,\omega_4,\omega_5,\omega_6)^{\top}$ is the vector of angular velocities. 

\subsection{Dynamical equations of the generator system\label{section 2.2}}
Regard $q=\left(\boldsymbol{\Psi};\boldsymbol{\theta}\right)$ as the generalized coordinates and $\dot{q}=\left(\dot{\boldsymbol{\Psi}};\dot{\boldsymbol{\theta}}\right)$ as the generalized velocity (here the semicolons mean that components are arrayed in column, keeping the same meaning in the remaining part of this article), then the dynamical equations of the generator system are given as follows.

Through the analysis in \cite{DTP,DMMS}, it is shown that the Lagrangian of this generator system has the following form 
$$L(\dot{\boldsymbol{\Psi}},\dot{\boldsymbol{\theta}},\boldsymbol{\Psi},\boldsymbol{\theta})=\left(\frac{1}{2}\dot{\boldsymbol{\Psi}}^{\top}K_C \dot{\boldsymbol{\Psi}}+\frac{1}{2}\dot{\boldsymbol{\theta}}^{\top} J\dot{\boldsymbol{\theta}}\right)-\left(\frac{1}{2}\boldsymbol{\Psi}^{\top}(K_L+\Gamma(\boldsymbol{\theta})) \boldsymbol{\Psi}+\frac{1}{2}\boldsymbol{\theta}^{\top} K\boldsymbol{\theta}\right),$$
where $J,K$ are the inertia matrix and the stiffness matrix of the mechanical shaft respectively, both of them are symmetric; $K_C,K_L$ are the coefficient matrices used for computing the capacitive energy and the inductive energy in the circuit part. Furthermore, $\Gamma(\boldsymbol{\theta})$ is a coefficient matrix used for calculating magnetic energy of the generator. For convenience, the specific forms of these matrices are provided here for future usage: 
\begin{align}\label{Eq. old coefficient 1}
&J=\mathrm{diag}(J_1,J_2,J_3,J_4,J_5,J_6),\ K_C=0, \notag \\
&K_L=\left(\begin{array}{cccccc}
          L^{-1} & 0 & -L^{-1} & 0 & 0 & 0 \\
          0 & L^{-1} & 0 & -L^{-1} & 0 & 0 \\
          -L^{-1} & 0 & L^{-1} & 0 & 0 & 0 \\
          0 & -L^{-1} & 0 & L^{-1} & 0 & 0 \\
          0 & 0 & 0 & 0 & 0 & 0 \\
          0 & 0 & 0 & 0 & 0 & 0 \\
          \end{array}
    \right), \notag \\
&K=\left(\begin{array}{cccccc}
K_1 & -K_1 & 0 & 0 & 0 & 0 \\
-K_1 & K_1+K_2	& -K_2 & 0 & 0 & 0 \\
0 & -K_2 & K_2+K_3 & -K_3 & 0 & 0 \\
0 & 0 & -K_3 & K_3+K_4 & -K_4 & 0 \\
0 & 0 & 0 & -K_4 & K_4+K_5 & -K_5 \\
0 & 0 & 0 & 0 & -K_5 & K_5 \\ 
\end{array}\right), \\
&\Gamma(\boldsymbol{\theta})=\frac{1}{\frac{3}{2}M^2-L_r(L_s+M_s)} \notag \\
&\left(\begin{array}{cccccc}
0 & 0 & 0 & 0 & 0 & 0 \\
0 & 0 & 0 & 0 & 0 & 0 \\
0 & 0 & -L_r & 0 & \sqrt{\frac{3}{2}}M\cos(\theta_5) & -\sqrt{\frac{3}{2}}M\sin(\theta_5)\\
0 & 0 & 0 & -L_r & \sqrt{\frac{3}{2}}M\sin(\theta_5) & \sqrt{\frac{3}{2}}M\cos(\theta_5)\\
0 & 0 & \sqrt{\frac{3}{2}}M\cos(\theta_5) & \sqrt{\frac{3}{2}}M\sin(\theta_5) & -L_s-M_s & 0\\
0 & 0 & -\sqrt{\frac{3}{2}}M\sin(\theta_5) & \sqrt{\frac{3}{2}}M\cos(\theta_5) & 0 & -L_s-M_s \\
\end{array}\right). \notag
\end{align}
Here $J_1=1166.56,J_2=1953.83,J_3=10782.84,J_4=11103.62,J_5=10906.22,J_6=429.68$ are the rotary inertia coefficients of the mechanical shaft, and $K_1=45692300.27,K_2=82680741.64,K_3=123179605.30,K_4=167728592$,\ $K_5=6679980.902$ are the stiffness factors. Moreover, $M=33.35 \mathrm{mH}$, $L_r=519 \mathrm{mH}, L_s=3 \mathrm{mH}, M_s=0.516 \mathrm{mH}$ are the parameters of the generator.

On the other hand, the Rayleigh's dissipation function of the whole system is
$$\mathcal{R}=\frac{1}{2}\dot{\boldsymbol{\Psi}}^{\top}K_R \dot{\boldsymbol{\Psi}}+\frac{1}{2}\dot{\boldsymbol{\theta}}^{\top}D \dot{\boldsymbol{\theta}}-\dot{\boldsymbol{\theta}}^{\top}T-\dot{\boldsymbol{\Psi}}^{\top}I_s(t),$$
where $K_R,D$ are the coefficient matrices used for describing the resistance loss and the friction loss respectively, and $T$ represents the vector composed of the mechanical torques $T_1,\cdots,T_4$ in Figure \ref{Fig:ACSGS}(c), $I_s(t)$ represents the vector describing the node-injected currents. For the same reason, we list the specific forms of $K_R,T,I_s(t)$ here for future reference:  
\begin{align}\label{Eq. old coefficient 2}
K_R=\mathrm{diag}\left(R^{-1},R^{-1},0,0,R_f^{-1},R_q^{-1}\right),
T=\left(\begin{array}{c}
     T_1  \\
     T_2  \\
     T_3  \\
     T_4  \\
     0    \\
     0
\end{array}
\right),
I_s(t)=\left(\begin{array}{c}
     U_s/R\cos(\omega_s t)  \\
     U_s/R\sin(\omega_s t)  \\
     0 \\
     0  \\
     U_f/R_f \\
     0
\end{array}
\right).
\end{align}
It is noteworthy that $R_f=R_q=0.1597\ \Omega, \omega_s=120\pi \ \mathrm{rad/s}$. Additionally, $T_1=601469.26$,	$T_2=521273.35,T_3=441077.45,T_4=441077.45$, and $I_f=\frac{U_f}{R_f}=3212.64\ \mathrm{A}$ is the exciting current remaining constant. 

So far, according to the Euler-Lagrange equation containing Rayleigh's dissipation function
$$\frac{\mathrm{d}}{\mathrm{d}t}\left(\frac{\partial L}{\partial \dot{q}}\right)-\frac{\partial L}{\partial q}+\frac{\partial \mathcal{R}}{\partial \dot{q}}=0,$$
the dynamical equations of the generator system can be derived, which read  
\begin{align}\label{Eq. old system}
\left\{\begin{aligned}
&K_C\ddot{\boldsymbol{\Psi}}+K_R\dot{\boldsymbol{\Psi}}+(K_L+\Gamma(\boldsymbol{\theta}))\boldsymbol{\Psi}=I_s(t), \\
&J\ddot{\boldsymbol{\theta}}+D\dot{\boldsymbol{\theta}}+K\boldsymbol{\theta}+\frac{1}{2}\boldsymbol{\Psi}^{\top} \frac{\partial\Gamma(\boldsymbol{\theta})}{\partial \boldsymbol{\theta}} \boldsymbol{\Psi}=T,
\end{aligned}\right.
\end{align}
here $\boldsymbol{\Psi}^{\top} \frac{\partial\Gamma(\boldsymbol{\theta})}{\partial \boldsymbol{\theta}} \boldsymbol{\Psi}:=\left(0,0,0,0,\boldsymbol{\Psi}^{\top} \frac{\partial\Gamma(\boldsymbol{\theta})}{\partial \theta_5} \boldsymbol{\Psi},0\right)^{\top}$ is an abbreviation, since $\Gamma(\boldsymbol{\theta})$ actually only depends on $\theta_5$. 

\section{Predictor-corrector methods\label{section 3}}
In this section, we present predictor-corrector methods (abbreviated to P-C methods) for the generator system, basing on the actual physical phenomena. Inspired by the reality that the rotary speeds of six mass blocks in Figure \ref{Fig:ACSGS}(c) vary slowly within a time step of the electromagnetic transient calculation, P-C methods can be proposed by applying Euler method to the vector of angles. To derive these methods, we firstly transform the dynamical equations (\ref{Eq. old system}) into 
\begin{align}\label{Eq. system for P-C}
\left\{\begin{aligned}
&K_{E_1}\dot{x}_E=-K_{E_2}(\boldsymbol{\theta})x_E+g_E(t),\\     
&K_{M_1}\dot{x}_M=-K_{M_2}x_M+g_M\left(\boldsymbol{\Psi},\boldsymbol{\theta}\right),
\end{aligned}
\right.     
\end{align}
where $x_E=\left(\dot{\boldsymbol{\Psi}}; \boldsymbol{\Psi}\right),\ x_M=\left(\dot{\boldsymbol{\theta}}; \boldsymbol{\theta}\right)$ and 
\begin{align}\label{Eq. K_E1, K_E2, K_M1, K_M2}
\begin{split}
&K_{E_1}=\mathrm{diag}(K_C,I_6),\ K_{E_2}(\boldsymbol{\theta})
                            =\left(\begin{array}{cc}
                            K_R & K_L+\Gamma(\boldsymbol{\theta}) \\
                            -I_6 & 0 
                            \end{array}\right),\      
g_E(t)=\left(\begin{array}{c} I_s(t) \\ 0 \end{array}\right), \\
&K_{M_1}=\mathrm{diag}(J,I_6),\ K_{M_2}=\left(\begin{array}{cc}
                                           D & K \\
                                           -I_6 & 0 
                                           \end{array}\right),\ 
g_M\left(\boldsymbol{\Psi},\boldsymbol{\theta}\right)=\left(\begin{array}{c}
T-\frac{1}{2}\boldsymbol{\Psi}^{\top} \frac{\partial \Gamma(\boldsymbol{\theta})}{\partial \boldsymbol{\theta}} \boldsymbol{\Psi} \\ 0 \end{array}\right)
\end{split}
\end{align}
with the matrices $J,K,D,\Gamma(\boldsymbol{\theta})$ and vector $T$ defined in Section \ref{section 2.2}. Let $x_{E,n}=\left(\dot{\boldsymbol{\Psi}}_n; \boldsymbol{\Psi}_n\right)$, $x_{M,n}=\left(\dot{\boldsymbol{\theta}}_n; \boldsymbol{\theta}_n\right)$ be the numerical solution to (\ref{Eq. system for P-C}) at $t_n=nh$, where $h$ represents the time step. After this, we calculate the numerical solution at $t_{n+1}=(n+1)h$ on the basis of the former solution.

Applying trapezoidal rule to the first equation of (\ref{Eq. system for P-C}) yields

\begin{align*}
K_{E_1}x_{E,n+1}=
&K_{E_1}x_{E,n}+\frac{h}{2}K_{E_1}\left(\dot{x}_{E,n}+\dot{x}_{E,n+1}\right) \\
=&K_{E_1}x_{E,n}+\frac{h}{2}\left[\left(-K_{E_2}(\boldsymbol{\theta}_n)x_{E,n}+g_E(t_n)\right)\right. \\
&\left.+\left(-K_{E_2}(\boldsymbol{\theta}_{n+1})x_{E,n+1}+g_E(t_{n+1})\right)
\right], \\
\end{align*}
which is equivalent to the following formula
\begin{align}\label{Eq. Trapezoidal rule}
\begin{split}
&\left[K_{E_1}+\frac{h}{2}K_{E_2}(\boldsymbol{\theta}_{n+1})\right]x_{E,n+1} \\
=&\left[K_{E_1}-\frac{h}{2}K_{E_2}(\boldsymbol{\theta}_n)\right]x_{E,n}+\frac{h}{2}\left(g_E(t_n)+g_E(t_{n+1})\right).  
\end{split}  
\end{align}
Therefore, it is obvious that only when $\boldsymbol{\theta}_{n+1}$ is known can $x_{E,n+1}$ be directly calculated. For the reason of this, we predict $\boldsymbol{\theta}_{n+1}$ with the assistance of forward Euler method, which gives  
\begin{align}\label{Eq. predictor}
\boldsymbol{\theta}_{n+1}^{[0]}=\boldsymbol{\theta}_n+h\dot{\boldsymbol{\theta}}_n.    
\end{align}
Combine (\ref{Eq. Trapezoidal rule}) and (\ref{Eq. predictor}), we obtain the formula to calculate $x_{E,n+1}$ as below 
\begin{align}\label{Eq. predictor Psi}
\begin{split}
&\left[K_{E_1}+\frac{h}{2}K_{E_2}\left(\boldsymbol{\theta}_{n+1}^{[0]}\right)\right]x_{E,n+1} \\
=&\left[K_{E_1}-\frac{h}{2}K_{E_2}(\boldsymbol{\theta}_n)\right]x_{E,n}+\frac{h}{2}\left(g_E(t_n)+g_E(t_{n+1})\right). 
\end{split}
\end{align}
Similarly, applying trapezoidal rule to the second equation of (\ref{Eq. system for P-C}) yields
\begin{align*}
\begin{split}
&\left(K_{M_1}+\frac{h}{2}K_{M_2}\right)x_{M,n+1} \\
=&\left(K_{M_1}-\frac{h}{2}K_{M_2}\right)x_{M,n}+\frac{h}{2}\left(g_M\left(\boldsymbol{\Psi}_n,\boldsymbol{\theta}_n\right)+g_M\left(\boldsymbol{\Psi}_{n+1},\boldsymbol{\theta}_{n+1}\right)\right),  \end{split}
\end{align*}
then different selections for evaluating $g_M\left(\boldsymbol{\Psi}_{n+1},\boldsymbol{\theta}_{n+1}\right)$ will lead to different methods to correct the numerical solution $\boldsymbol{\theta}_{n+1}^{[0]}$. As a result, F. Ji and C. Lin proposed the following P-C method (I) by taking $g_M\left(\boldsymbol{\Psi}_{n+1},\boldsymbol{\theta}_{n+1}\right)\approx g_M\left(\boldsymbol{\Psi}_n,\boldsymbol{\theta}_n\right)$:
\begin{subequations}
\begin{align}
(i)&\ \boldsymbol{\theta}_{n+1}^{[0]}=\boldsymbol{\theta}_n+h\dot{\boldsymbol{\theta}}_n, \label{Eq. P-C I a}\\
(ii)&\left[K_{E_1}+\frac{h}{2}K_{E_2}\left(\boldsymbol{\theta}_{n+1}^{[0]}\right)\right]x_{E,n+1} \notag \\
=&\left[K_{E_1}-\frac{h}{2}K_{E_2}(\boldsymbol{\theta}_n)\right]x_{E,n}+\frac{h}{2}\left(g_E(t_n)+g_E(t_{n+1})\right), \label{Eq. P-C I b}\\
(iii)&\left(K_{M_1}+\frac{h}{2}K_{M_2}\right)x_{M,n+1}
=\left(K_{M_1}-\frac{h}{2}K_{M_2}\right)x_{M,n}+hg_M\left(\boldsymbol{\Psi}_n,\boldsymbol{\theta}_n\right). \label{Eq. P-C I c}     
\end{align}    
\end{subequations}
On the other side, since $\boldsymbol{\theta}_{n+1}^{[0]},\boldsymbol{\Psi}_{n+1}$ have been attained from (\ref{Eq. P-C I a}) and (\ref{Eq. P-C I b}), respectively, there exists a more reasonable approach to evaluating $g_M\left(\boldsymbol{\Psi}_{n+1},\boldsymbol{\theta}_{n+1}\right)$ with the assistance of $\boldsymbol{\Psi}_{n+1},\boldsymbol{\theta}_{n+1}^{[0]}$, which gives 
\begin{align}\label{Eq. P-C II}
\begin{split}
&\left(K_{M_1}+\frac{h}{2}K_{M_2}\right)x_{M,n+1} \\
=&\left(K_{M_1}-\frac{h}{2}K_{M_2}\right)x_{M,n}+\frac{h}{2}\left(g_M\left(\boldsymbol{\Psi}_n,\boldsymbol{\theta}_n\right)+g_M\left(\boldsymbol{\Psi}_{n+1},\boldsymbol{\theta}_{n+1}^{[0]}\right)\right).    
\end{split}    
\end{align}
Therefore, we obtain another numerical method composed of (\ref{Eq. P-C I a}), (\ref{Eq. P-C I b}) and (\ref{Eq. P-C II}), and we call it P-C method (II). 

Subsequent numerical experiments will illustrate the efficiency and accuracy of these P-C methods, which indicates that prior physical facts are beneficial to construct effective numerical methods.

\section{Structure-preserving methods \label{section 4}}
In this section, we reformulate the generator system from the perspective of structure-preserving dynamical systems, and  construct its structure-preserving algorithm. Due to the energy dissipation in generator system caused by resistance loss and friction loss, it is not suitable to model this system using the symplectic structure of Hamiltonian systems. Therefore, we turn to a more generalized framework known as the port-Hamiltonian descriptor system introduced in \cite{pHDAE},  which is able to account for the energy dissipation effects and is associated with a Dirac structure. We will show that the generator system can be reformulated as a port-Hamiltonian descriptor system, and thereby employ structure-preserving methods.

\subsection{Port-Hamiltonian form for the generator system\label{section 4.1}}
Recall that the coefficient matrices $K_C=0$ and $K_R$ is singular (see (\ref{Eq. old coefficient 1})$\sim$(\ref{Eq. old coefficient 2})), so $\dot{\boldsymbol{\Psi}}$ can not be directly expressed by $t,\boldsymbol{\Psi},\boldsymbol{\theta}$ from the first equation of (\ref{Eq. old system}). As a consequence, it is worth trying to reformulate the dynamical equations (\ref{Eq. old system}) to eliminate the singularity of $K_R$. Observation indicates that the derivatives of $\Psi_{2\alpha},\Psi_{2\beta}$ are not involved in these equations, so it is natural to express $\Psi_{2\alpha},\Psi_{2\beta}$ by other variables, which leads to
\begin{align}\label{Eq. transfer}
\begin{split}
&\Psi_{2\alpha}=\lambda_1 \Psi_{1\alpha}+\lambda_2 (-\Psi_f \cos\theta_5+\Psi_q \sin\theta_5),\\
&\Psi_{2\beta}=\lambda_1 \Psi_{1\beta}+\lambda_2 (-\Psi_f \sin\theta_5-\Psi_q \cos\theta_5), \end{split}
\end{align}
where the specific forms of the coefficients in (\ref{Eq. transfer}) are $\lambda_1=\frac{\frac{3}{2}M^2-L_r(L_s+M_s)}{\frac{3}{2}M^2-L_r(L_s+M_s+L)}, \lambda_2=\frac{\sqrt{\frac{3}{2}}ML}{\frac{3}{2}M^2-L_r(L_s+M_s+L)}$. Insert (\ref{Eq. transfer}) into (\ref{Eq. old system}) and take
\begin{align*}
\widetilde{\boldsymbol{\Psi}}=(\Psi_{1\alpha},\Psi_{1\beta},\Psi_f,\Psi_q)^{\top}   
\end{align*}
as the new vector of flux linkages, then we obtain 
\begin{align}\label{Eq. new system}
\left\{\begin{aligned}
&\widetilde{K}_R \dot{\widetilde{\boldsymbol{\Psi}}}+\left(\widetilde{K}_L+\widetilde{\Gamma}(\boldsymbol{\theta})\right) \widetilde{\boldsymbol{\Psi}}=\widetilde{I}_s(t),\\     
&J\ddot{\boldsymbol{\theta}}+D\dot{\boldsymbol{\theta}}+K\boldsymbol{\theta}+\frac{1}{2}\widetilde{\boldsymbol{\Psi}}^{\top} \frac{\partial \widetilde{\Gamma}(\boldsymbol{\theta})}{\partial \boldsymbol{\theta}} \widetilde{\boldsymbol{\Psi}}=T,
\end{aligned}
\right.    
\end{align}
where $\widetilde{\boldsymbol{\Psi}}^{\top} \frac{\partial \widetilde{\Gamma}(\boldsymbol{\theta})}{\partial \boldsymbol{\theta}} \widetilde{\boldsymbol{\Psi}}=\left(0,0,0,0,\widetilde{\boldsymbol{\Psi}}^{\top} \frac{\partial \widetilde{\Gamma}(\boldsymbol{\theta})}{\partial \theta_5} \widetilde{\boldsymbol{\Psi}},0\right)^{\top}$, and
\begin{align}\label{Eq. new coefficient}
\begin{split}
&\widetilde{K}_R=\mathrm{diag}\left(R^{-1},R^{-1},R_f^{-1},R_q^{-1}\right),\widetilde{I}_s(t)=\left(\frac{U_s}{R}\cos(\omega_s t),\frac{U_s}{R}\sin(\omega_s t),\frac{U_f}{R_f},0\right)^{\top},\\ 
&\widetilde{K}_L=\frac{\mathrm{diag}\left(-L_r,-L_r,-(L_s+M_s+L),-(L_s+M_s+L)\right)}{\frac{3}{2}M^2-L_r(L_s+M_s+L)},\\
&\widetilde{\Gamma}(\boldsymbol{\theta})=\frac{\sqrt{\frac{3}{2}}M}{\frac{3}{2}M^2-L_r(L_s+M_s+L)}\left(\begin{array}{cccc}
    0 & 0 & \cos(\theta_5) & -\sin(\theta_5) \\
    0 & 0 & \sin(\theta_5) & \cos(\theta_5)  \\
    \cos(\theta_5) & \sin(\theta_5) & 0 & 0  \\
    -\sin(\theta_5) & \cos(\theta_5) & 0 & 0 \\
\end{array}\right).
\end{split}    
\end{align}

Let 
\begin{align}\label{Eq. state, input, output}
x(t)=\left(\begin{array}{c} \dot{\widetilde{\boldsymbol{\Psi}}} \\ \widetilde{\boldsymbol{\Psi}} \\ \dot{\boldsymbol{\theta}} \\ \boldsymbol{\theta} \\ t \end{array}\right),\ u(x)=\left(\begin{array}{c} \widetilde{I}_s(t) \\ 0 \\ T \\ 0 \\ 1
\end{array}\right),\ y(x)=\left(\begin{array}{c}
\dot{\widetilde{\boldsymbol{\Psi}}} \\ 0 \\ \dot{\boldsymbol{\theta}} \\ 0 \\ 0 \end{array}\right) 
\end{align}
be the state, input and output, respectively, then the generator system (\ref{Eq. new system}) can be written in the the following form
\begin{align}\label{Eq. pHDAE}
\begin{split}
M\dot{x}&=(P-Q)z(x)+(N-V)u(x), \\
y&=(N+V)^{\top}z(x)+(S-W)u(x),
\end{split}
\end{align}
where the coefficient matrices are given as
\begin{align}\label{Eq. matrices of pHDS}
\begin{split}
&M=\mathrm{diag}\left(0,\widetilde{K}_L,J,K,1\right),\ N=\mathrm{diag}(I_4,0,I_6,0,1),\ V=S=W=0, \\  
&P=\left(\begin{array}{ccccc}
    0 & -\widetilde{K}_L & 0 & 0 & 0 \\
    \widetilde{K}_L & 0 & 0 & 0 & 0 \\
    0 & 0 & 0 & -K & 0 \\
    0 & 0 & K & 0 & 0 \\
    0 & 0 & 0 & 0 & 0
\end{array}\right),\ Q=\mathrm{diag}\left(\widetilde{K}_R,0,D,0,0\right)
\end{split}    
\end{align}
and 
\begin{align}\label{Eq. z}
z(x)=\left(\begin{array}{c}
\dot{\widetilde{\boldsymbol{\Psi}}} \\ \widetilde{K}_L^{-1}\left(\widetilde{K}_L+\widetilde{\Gamma}(\boldsymbol{\theta})\right)\widetilde{\boldsymbol{\Psi}} \\
\dot{\boldsymbol{\theta}} \\
K^{-1}\left(K\boldsymbol{\theta}+\frac{1}{2}\widetilde{\boldsymbol{\Psi}}^{\top} \frac{\partial \widetilde{\Gamma}(\boldsymbol{\theta})}{\partial \boldsymbol{\theta}} \widetilde{\boldsymbol{\Psi}}\right) \\ 0
\end{array}\right).
\end{align}
In addition, by regarding $\left(\widetilde{\boldsymbol{\Psi}};\boldsymbol{\theta}\right)$ as the generalized coordinates and $\left(\dot{\widetilde{\boldsymbol{\Psi}}};\dot{\boldsymbol{\theta}}\right)$ as the generalized velocity, system (\ref{Eq. pHDAE}) has a Hamiltonian function
\begin{align}\label{Eq. new Hamiltonian}
\mathcal{H}(x)=\frac{1}{2}\dot{\boldsymbol{\theta}}^{\top} J\dot{\boldsymbol{\theta}}+\frac{1}{2}\widetilde{\boldsymbol{\Psi}}^{\top}(\widetilde{K}_L+\widetilde{\Gamma}(\boldsymbol{\theta})) \widetilde{\boldsymbol{\Psi}}+\frac{1}{2}\boldsymbol{\theta}^{\top} K\boldsymbol{\theta}.
\end{align}
Moreover, we can readily check that the matrix functions
\begin{align}\label{Eq. matrix functions}
\Xi:=\left(\begin{array}{cc}
    P & N \\
    -N^{\top} & W 
\end{array}\right),\ 
\Lambda:=\left(\begin{array}{cc}
    Q & V \\
    V^{\top} & S 
\end{array}\right)\ \text{satisfy $\Xi=-\Xi^{\top}$ and $\Lambda=\Lambda^{\top}\geq 0$};
\end{align}
and 
\begin{equation}\label{Eq. hami function}
\nabla_x \mathcal{H}=M^{\top}z. 
\end{equation}
In this way, according to \cite[Definition 1]{pHDAE}, the generator system (\ref{Eq. new system}) is actually an autonomous port-Hamiltonian descriptor system, which is associated with a Dirac structure. 
\begin{remark}\label{Rmk. PBE}
In fact, two necessary conditions (\ref{Eq. matrix functions}) and (\ref{Eq. hami function}) result in a power balance equation of the Hamiltonian function (\ref{Eq. new Hamiltonian})
\begin{align*}
\frac{\mathrm{d}}{\mathrm{d}t}\mathcal{H}(x(t))=-\left(\begin{array}{c} z \\ u \end{array}\right)^{\top}\Lambda \left(\begin{array}{c} z \\ u \end{array}\right)+y^{\top}u\leq y^{\top}u,  
\end{align*}
which holds along any solution $x(t)$ and for any input $u(x)$. Additionally, there exists an inequality
$$\mathcal{H}(x(t_f))-\mathcal{H}(x(t_0))\leq \int_{t_0}^{t_f} y(\tau)^{\top}u(\tau)d\tau,$$
which could be implemented to evaluate the energy dissipation. 
\end{remark}

\subsection{Structure-preserving methods\label{section 4.2}}
Now we consider to apply Runge-Kutta methods to the generator system in port-Hamiltonian form, yielding that
\begin{align}\label{Eq. collocation for system}
\begin{split}
Mk_i&=(P-Q)z\left(x_0+h\sum\limits_{j=1}^s a_{ij}k_j\right)+(N-V)u\left(x_0+h\sum\limits_{j=1}^s a_{ij}k_j\right), \\
x_f&=x_0+h\sum\limits_{j=1}^s b_j k_j.
\end{split}
\end{align}
Here we employ the collocation method with the coefficients $a_{ij},b_i$ taking the form
\begin{align}\label{Eq. collocation coefficient}
a_{ij}=\int_0^{\gamma_i}\ell_j(\tau)d\tau,\ b_i=\int_0^1\ell_i(\tau)d\tau,    
\end{align}
where $\ell_i(\tau)=\prod\limits_{l=1\atop l\neq i}^s \frac{\tau-\gamma_l}{\gamma_i-\gamma_l}$ is the Lagrange interpolation polynomial, and $\gamma_1,\cdots,\gamma_s$ are distinct real numbers located in $[0,1]$. In particular, if $\gamma_1,\cdots,\gamma_s$ are the zero points of the $s$-th shifted Legendre polynomial
\begin{equation*}
\frac{\mathrm{d}^s}{\mathrm{d}x^s}\bigg(x^s(x-1)^s\bigg),    
\end{equation*}
then we obtain $s$-stage Gauss method. Here we list Gauss methods for $s=1,2,3$ in the following Table \ref{Tab:Gauss methods} as instance. 
\begin{table}[ht]
\caption{Gauss methods for $s=1,2,3$.}
\label{Tab:Gauss methods}
\renewcommand{\arraystretch}{1.5}
\centering
\subtable[$s=1$]{
   \centering
   \begin{tabular}{c|c}
   $\frac{1}{2}$ & $\frac{1}{2}$ \\
   \hline
               &    1
   \end{tabular}
}  
\subtable[$s=2$]{
\centering
   \begin{tabular}{c|cc}
   $\frac{1}{2}-\frac{\sqrt{3}}{6}$ & $\frac{1}{4}$ & $\frac{1}{4}-\frac{\sqrt{3}}{6}$ \\
   $\frac{1}{2}+\frac{\sqrt{3}}{6}$ & $\frac{1}{4}+\frac{\sqrt{3}}{6}$ & $\frac{1}{4}$ \\
   \hline
               &   $\frac{1}{2}$ & $\frac{1}{2}$
   \end{tabular}    
}
\subtable[$s=3$]{
\centering
   \begin{tabular}{c|ccc}
   $\frac{1}{2}-\frac{\sqrt{15}}{10}$ & $\frac{5}{36}$ & $\frac{2}{9}-\frac{\sqrt{15}}{15}$ & $\frac{5}{36}-\frac{\sqrt{15}}{30}$\\
   $\frac{1}{2}$ & $\frac{5}{36}+\frac{\sqrt{15}}{24}$ & $\frac{2}{9}$ & $\frac{5}{36}-\frac{\sqrt{15}}{24}$\\
   $\frac{1}{2}+\frac{\sqrt{15}}{10}$ & $\frac{5}{36}+\frac{\sqrt{15}}{30}$ & $\frac{2}{9}+\frac{\sqrt{15}}{15}$ & $\frac{5}{36}$\\
   \hline
               &  $\frac{5}{18}$ & $\frac{4}{9}$ & $\frac{5}{18}$
   \end{tabular}    
}
\end{table}

Notice that elements in the first four rows of $M$ given by (\ref{Eq. matrices of pHDS}) are all zeros, which means the generator system in port-Hamiltonian form is actually a differential-algebraic system, and $k_i$ can not be directly expressed by $k_j\ (j=1,\cdots,s)$ from the first equation in (\ref{Eq. collocation for system}). In this case, the global error of (\ref{Eq. collocation for system}) can be deduced by the theory of differential-algebraic equations. For instance, we have the following theorem aiming at Gauss methods.
\begin{theorem}\label{Thm. global error of Gauss method}
Apply $s$-stage Gauss method to the generator system (\ref{Eq. pHDAE}), whose specific form is (\ref{Eq. collocation for system}). Assume that the initial values are consistent, then the global error of the numerical solution given by (\ref{Eq. collocation for system}) satisfies
\begin{align*}
&\widetilde{\boldsymbol{\Psi}}_n-\widetilde{\boldsymbol{\Psi}}(t_n)=\mathcal{O}(h^{2s}),\ 
\dot{\boldsymbol{\theta}}_n-\dot{\boldsymbol{\theta}}(t_n)=\mathcal{O}(h^{2s}),\
\boldsymbol{\theta}_n-\boldsymbol{\theta}(t_n)=\mathcal{O}(h^{2s}), \\
&\dot{\widetilde{\boldsymbol{\Psi}}}_n-\dot{\widetilde{\boldsymbol{\Psi}}}(t_n)=\left\{ \begin{aligned}
   &\mathcal{O}(h^s), && \hbox{for even}\ s,  \\
   &\mathcal{O}(h^{s+1}), && \hbox{for odd}\ s.
\end{aligned}\right.   
\end{align*}
for $t_n=nh\leq T_m$.
\end{theorem}
\begin{proof}
We give the proof in Appendix \ref{appendix}.
\end{proof}

\subsubsection{Dirac-structure preservation for the generator system\label{section 4.2.1}}
Usually, port-Hamiltonian systems are described by Dirac structures which can be considered as the generalizations of symplectic structures \cite[Section \uppercase\expandafter{\romannumeral2}.C ]{EPPC}. The fundamental property of a Dirac structure manifests itself in power conservation \cite[ Section 2.2]{PHST}, which means the Dirac structure connects the port variables $v_f,v_e$ in a way that the total power $v_e^{\top}v_f=0$, here $v_f,v_e$ represent the flow variable and effort variable, respectively. In this subsection, we introduce the definition of Dirac structure following \cite[Section 5.1]{PHST}, and then show that collocation methods employed for the generator system preserve the discrete Dirac structure at all collocation points $x_i$.

\begin{definition}[Linear Dirac structure]\label{Def. linear Dirac structure}
Let $\mathcal{F}$ be an $n$-dimensional linear space of flows and $\mathcal{E}=\mathcal{F}^{\ast}$ be its dual space of efforts. In addition, $\mathcal{U}$ is another linear space of dimension $n$, $F,E$ are $n\times n$ matrices representing the linear maps $F:\mathcal{F}\to \mathcal{U}$ and $E:\mathcal{E}\to \mathcal{U}$, respectively. Therefore, a linear subspace
\begin{align}\label{Eq. D}
\mathcal{D}=\left\{(v_f,v_e)\in \mathcal{F}\times \mathcal{E}\ |\ Fv_f+Ev_e=0\right\}\subseteq \mathcal{F}\times \mathcal{E}    
\end{align}
is a Dirac structure, if the matrices $F,E$ satiesfy
\begin{align}\label{Eq. E,F}
\begin{split}
&(i)\  EF^{\top}+FE^{\top}=0, \\
&(ii)\ \mathrm{rank}(F\ |\ E)=n.
\end{split}
\end{align}
\end{definition}

(\ref{Eq. D}) is the \textit{matrix kernel} representation of Dirac structure, and several other representations are displayed in \cite[Section 5]{PHST}. On this basis, we introduce a more general definition associated with Dirac structure following the \cite[Definition 3]{pHDAE}, so as to depict the structure-preserving property of collocation methods.   
\begin{definition}[General Dirac structure]\label{Def. general Dirac structure}
Let $\mathcal{X}$ be a state space and $\mathcal{V}$ be a vector bundle over $\mathcal{X}$ with fibers $\mathcal{V}_x\ (x\in \mathcal{X})$. A Dirac structure on $\mathcal{V}$ is a vector sub-bundle $\mathcal{D}\subseteq \mathcal{V}\bigoplus \mathcal{V}^{\ast}$ such that
$$\mathcal{D}_x\subseteq \mathcal{V}_x\times \mathcal{V}_x^{\ast}$$
is a linear Dirac structure for every $x\in \mathcal{X}$.
\end{definition}
\begin{remark}
The notation $\bigoplus$ here means the Whitney sum of two vector bundles $\mathcal{V}$ and $\mathcal{V}^{\ast}$, which is defined as the vector bundle whose fiber over each $x\in \mathcal{X}$ is naturally the direct product of the fibers $\mathcal{V}_x$ and $\mathcal{V}_x^{\ast}$.    
\end{remark}

Now we can correlate a Dirac structure with the generator system presented in the form of (\ref{Eq. pHDAE}). In fact, the connection between the autonomous generator system and its Dirac structure over the state space $\mathcal{X}$ can be established by the similar way of \cite[Theorem 2]{pHDAE}. Here we give a brief description to illustrate this point.

For the generator system in port-Hamiltonian form, consider the state space $\mathcal{X}$ and a vector bundle $\mathcal{V}$ over it. Define the flow fiber $\mathcal{V}_{x}=\mathcal{F}_{x}^s\times \mathcal{F}_{x}^p\times \mathcal{F}_{x}^d$ for each $x\in \mathcal{X}$, where $\mathcal{F}_{x}^s:=MT_{x}\mathcal{X}\subseteq \mathbb{R}^{21},\mathcal{F}_{x}^p:=\mathbb{R}^{21},\mathcal{F}_{x}^d:=\mathbb{R}^{42}$ are the storage flow fiber, port flow fiber and dissipation flow fiber, respectively. Write $v_f\in \mathcal{V}$ in partitioned form $v_f=(v_f^s;v_f^p;v_f^d)$. Similarly, $v_e\in \mathcal{V}^{\ast}$ has the partitioned expression $v_e=(v_e^s;v_e^p;v_e^d)$, thus the sub-bundle $\mathcal{D}\subseteq \mathcal{V}\bigoplus \mathcal{V}^{\ast}$ with 
\begin{align}\label{Eq. D_x}
\mathcal{D}_{x}=\left\{(v_f,v_e)\in \mathcal{V}_{x}\times \mathcal{V}_{x}^{\ast}\ \Bigg|\ v_f+\left(\begin{array}{cc}
   \Xi & I_{42}  \\
   -I_{42} & 0 
\end{array}\right)v_e=0\right\}    
\end{align}
is a Dirac structure on $\mathcal{V}$. Additionally, let $v_f,v_e$ satisfy
\begin{align*}
&v_f^s=-M\dot{x},\ v_f^p=y,\ v_f^d=(z(x);u(x)), \\ 
&v_e^s=z(x),\ v_e^p=u(x),\ v_e^d=-\Lambda(z(x);u(x)),
\end{align*}
then the generator system is equivalent to $(v_f,v_e)\in \mathcal{D}_x$.

So far, it is clear that the generator system preserves a Dirac structure given by (\ref{Eq. D_x}) along the solution $x(t)$ and input $u(x)$. Inspired by \cite[Section \uppercase\expandafter{\romannumeral3}.B]{pHDAE}, collocation methods applied to the generator system preserve the discrete Dirac structure at all collocation points $x_i\ (i=1,\cdots,s)$ as follows. 

Let $\gamma_1,\cdots,\gamma_s$ be distinct numbers and define coefficients $a_{ij},b_j$ by (\ref{Eq. collocation coefficient}). Take $k_i\ (i=1,\cdots,s)$ as certain unknowns, then choose collocation points $x_i=x_0+h\sum\limits_{j=1}^s a_{ij}k_j$. Consequently, there exists discrete Dirac structure $\left\{\mathcal{D}_{x_i}\ |\ i=1,\cdots,s\right\}$ defined by 
\begin{align}\label{Eq. D_xi}
\mathcal{D}_{x_i}=\left\{(v_{f,i},v_{e,i})\in \mathcal{V}_{x_i}\times \mathcal{V}_{x_i}^{\ast}\ \Bigg|\ v_{f,i}+\left(\begin{array}{cc}
   \Xi & I_{42}  \\
   -I_{42} & 0 
\end{array}\right)v_{e,i}=0\right\}    
\end{align}
at all collocation points $x_i$. Additionally, let $v_{f,i},v_{e,i}$ satisfy
\begin{align*}
&v_{f,i}^s=-Mk_i,\ v_{f,i}^p=y(x_i),\ v_{f,i}^d=(z(x_i);u(x_i)), \\ 
&v_{e,i}^s=z(x_i),\ v_{e,i}^p=u(x_i),\ v_{e,i}^d=-\Lambda(z(x_i);u(x_i)),    
\end{align*}
thus applying collocation method to the generator system in port-Hamiltonian form, i.e. (\ref{Eq. collocation for system}), is equivalent to $(v_{f,i},v_{e,i})\in \mathcal{D}_{x_i}$ together with 
\begin{align*}
x_f=x_0+h\sum\limits_{j=1}^s b_j k_j.
\end{align*}

\begin{remark}\label{Rmk. Discrete PBE}
Let $H(t):=\mathcal{H}(\overline{x}(t))$, where $\overline{x}(t)$ is the collocation polynomial of the state $x(t)$. Therefore, according to the properties of collocation polynomial, the power balance equation
\begin{align*}
\dot{H}(t_0+\gamma_i h)&=\nabla \mathcal{H}(\overline{x}(t_0+\gamma_i h))^{\top}\dot{\overline{x}}(t_0+\gamma_i h)=\nabla \mathcal{H}(x_i)^{\top}k_i=z(x_i)^{\top}Mk_i \\
&=-(z(x_i);u(x_i))^{\top} \Lambda (z(x_i);u(x_i))+y(x_i)^{\top}u(x_i)\leq y(x_i)^{\top}u(x_i) 
\end{align*}
holds for $i=1,\cdots,s$. Apply the quadrature formula associated with this collocation method, then if the coefficients $b_j\geq 0$, we obtain 
\begin{align*}
\mathcal{H}(x_f)-\mathcal{H}(x_0)&=H(t_f)-H(t_0)=\int_{t_0}^{t_f} \dot{H}(\tau)d\tau \\
&=h\sum\limits_{j=1}^s b_j \dot{H}(t_0+\gamma_j h)+\mathcal{O}(h^{p+1})\leq h\sum\limits_{j=1}^s b_j y(x_j)^{\top}u(x_j)+\mathcal{O}(h^{p+1}),
\end{align*}
where $\mathcal{O}(h^{p+1})$ represents the remainder of the quadrature formula. In the case that the method is $s$-stage Gauss method and the Hamiltonian function is quadratic, the remainder $\mathcal{O}(h^{p+1})$ vanishes, which gives a discrete dissipation inequality similar to Remark \ref{Rmk. PBE}.
\end{remark}

\section{Numerical simulations\label{section 5}}
In this section, we present numerical simulations on the generator system. Actually, the generator system in Figure \ref{Fig:ACSGS} will instantly reach a steady state, where all the rotary speeds of six mass blocks in Figure \ref{Fig:ACSGS}(c) remain tightly close to $\omega_s=120\pi \ \mathrm{rad/s}$. Therefore, it is adequate to concentrate on the numerical errors of six angular velocities $\omega_1,\cdots,\omega_6$, which can significantly reflect the effectiveness of the numerical methods. 

Besides the parameters of the generator system provided in Section \ref{section 2}, we select the following consistent initial values 
\begin{align*}
\dot{\boldsymbol{\Psi}}_0=\left(\begin{array}{c}
26014.5269 \\ 1.9571 \\ 25102.2884 \\ 6773.1172 \\ 0 \\ 0
\end{array}\right), \
\boldsymbol{\Psi}_0=\left(\begin{array}{c}
  0.0052 \\ -69.0057 \\ 17.9663 \\ -66.5859 \\ 645.4103 \\ -624.0651
\end{array}\right),\ \dot{\boldsymbol{\theta}}_0=\left(\begin{array}{c}
  120\pi \\ 120\pi \\ 120\pi \\ 120\pi \\ 120\pi \\ 120\pi
\end{array}\right),\ \boldsymbol{\theta}_0=\left(\begin{array}{c}
  -0.3629 \\ -0.3761 \\ -0.3897 \\ -0.4024 \\ -0.4143 \\ -0.4143
\end{array}\right),   
\end{align*}
here all the data (except $0,120\pi$) are account to four decimal places. Especially, for structure-preserving methods we choose
\begin{align*}
\dot{\widetilde{\boldsymbol{\Psi}}}_0=\left(\begin{array}{c}
26014.5269 \\ 1.9571 \\ 0 \\ 0
\end{array}\right), \
\widetilde{\boldsymbol{\Psi}}_0=\left(\begin{array}{c}
  0.0052 \\ -69.0057 \\ 645.4103 \\ -624.0651
\end{array}\right).
\end{align*}
Furthermore, the coefficient matrix $D$ describing the friction loss will be set to $0$, which is the ideal situation but will not influence markedly on the numerical results. 


Choose time step $h=10^{-4}$ and take $1$-stage Gauss method for an example of structure-preserving method, then we display the simulation results of P-C methods and structure-preserving method for $0\sim 10\ \mathrm{s}$ in Figure \ref{fig:10s}, compared with those given by PSCAD/EMTDC. It is apparent that all the methods proposed in this article possess better performance over PSCAD/EMTDC in numerical simulation, because the results obtained from P-C methods and structure-preserving method converge rapidly to the equilibrium point after a short transient process, while PSCAD/EMTDC gives results with obviously larger fluctuations after the state switching to generator operation at around $2.5\ \mathrm{s}$.     

\begin{figure}[!ht]
 \centering
 \includegraphics[width=0.9\textwidth]{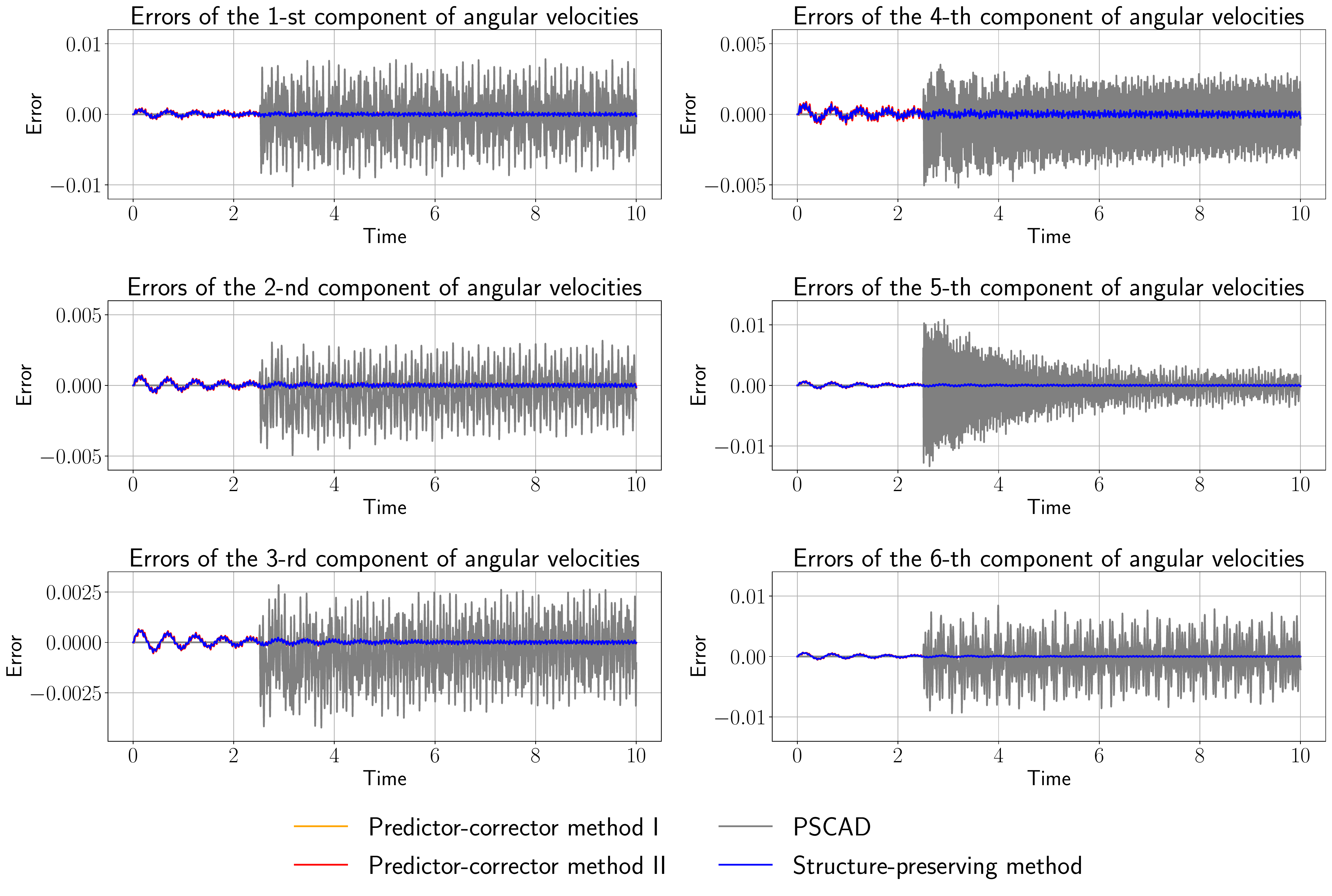}
 \caption{Simulation errors of angular velocities for $0\sim 10\ \mathrm{s}$. The errors here refer to the distinctions between the numerical results and the stable equilibrium point $\omega_s=120\pi \ \mathrm{rad/s}$.}
 \label{fig:10s}
\end{figure}

Next, we make the comparative analysis of P-C methods and structure-preserving method in long-term numerical simulation, whose results are shown in Figure \ref{fig:1000s}. In Figure \ref{fig:1000s}, we can see that both P-C method (II) and structure-preserving method have excellent computational stability, while the errors of P-C method (I) blow up at around $1000\ \mathrm{s}$. The reason for this phenomenon is that P-C method (I) is derived by left endpoint approximation, i.e. $g_M\left(\boldsymbol{\Psi}_{n+1},\boldsymbol{\theta}_{n+1}\right)\approx g_M\left(\boldsymbol{\Psi}_n,\boldsymbol{\theta}_n\right)$, neglecting $\boldsymbol{\Psi}_{n+1},\boldsymbol{\theta}_{n+1}^{[0]}$ obtained from (\ref{Eq. P-C I a}), (\ref{Eq. P-C I b}). On the contrary, P-C method (II) derived by $g_M\left(\boldsymbol{\Psi}_{n+1},\boldsymbol{\theta}_{n+1}\right)\approx g_M\left(\boldsymbol{\Psi}_{n+1},\boldsymbol{\theta}_{n+1}^{[0]}\right)$ shows excellent long-term stability in numerical simulation, which indicates that prior physical facts together with appropriate algorithm construction can lead to impressive simulation performance. 

Finally, we draw a detailed comparison between P-C method (II) and structure-preserving method in Figure \ref{fig:1500s}. From the simulation results, we can discover that the errors of structure-preserving method remain at around half the level of P-C method (II) throughout $10\sim 1500\ \mathrm{s}$. According to Theorem \ref{Thm. global error of Gauss method}, the structure-preserving method we take, i.e. $1$-stage Gauss method, is of order $2$, which can be verified by the simulation errors presented in Figure \ref{fig:1500s}. Moreover, since the construction of P-C method (II) is analogous to the improved Euler method, it is reasonable to treat this method as a second-order method. In this way, structure-preserving method exhibits its advantage over other numerical methods in long-term computational stability, as expected from its Dirac-structure preservation.

\begin{figure}[!ht]
 \centering
 \includegraphics[width=0.9\textwidth]{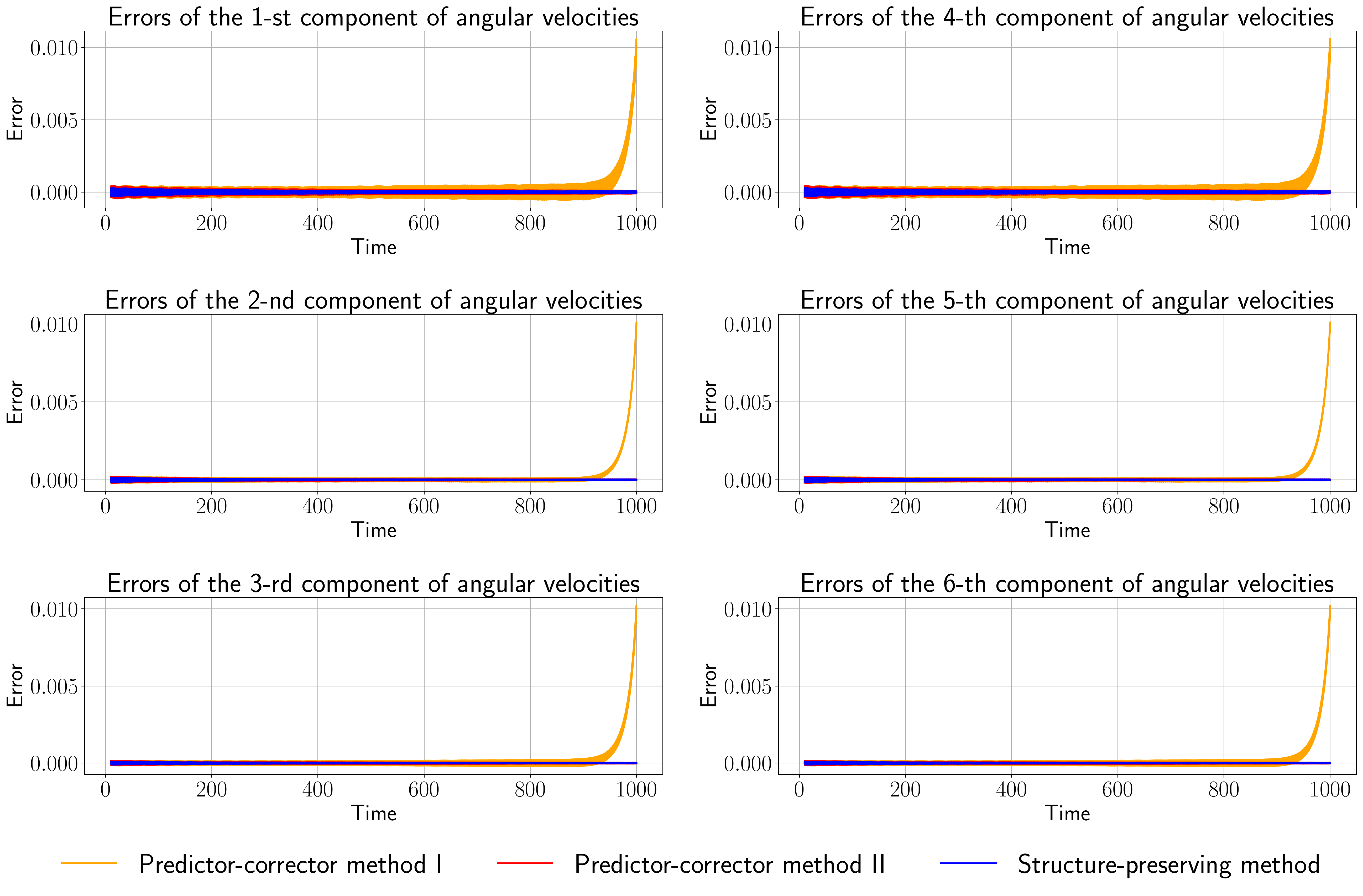}
 \caption{Simulation errors of angular velocities for $10\sim 1000\ \mathrm{s}$. The errors have the same meaning as those in Figure \ref{fig:10s}.}
 \label{fig:1000s}
\end{figure}

\begin{figure}[!ht]
 \centering
 \includegraphics[width=0.9\textwidth]{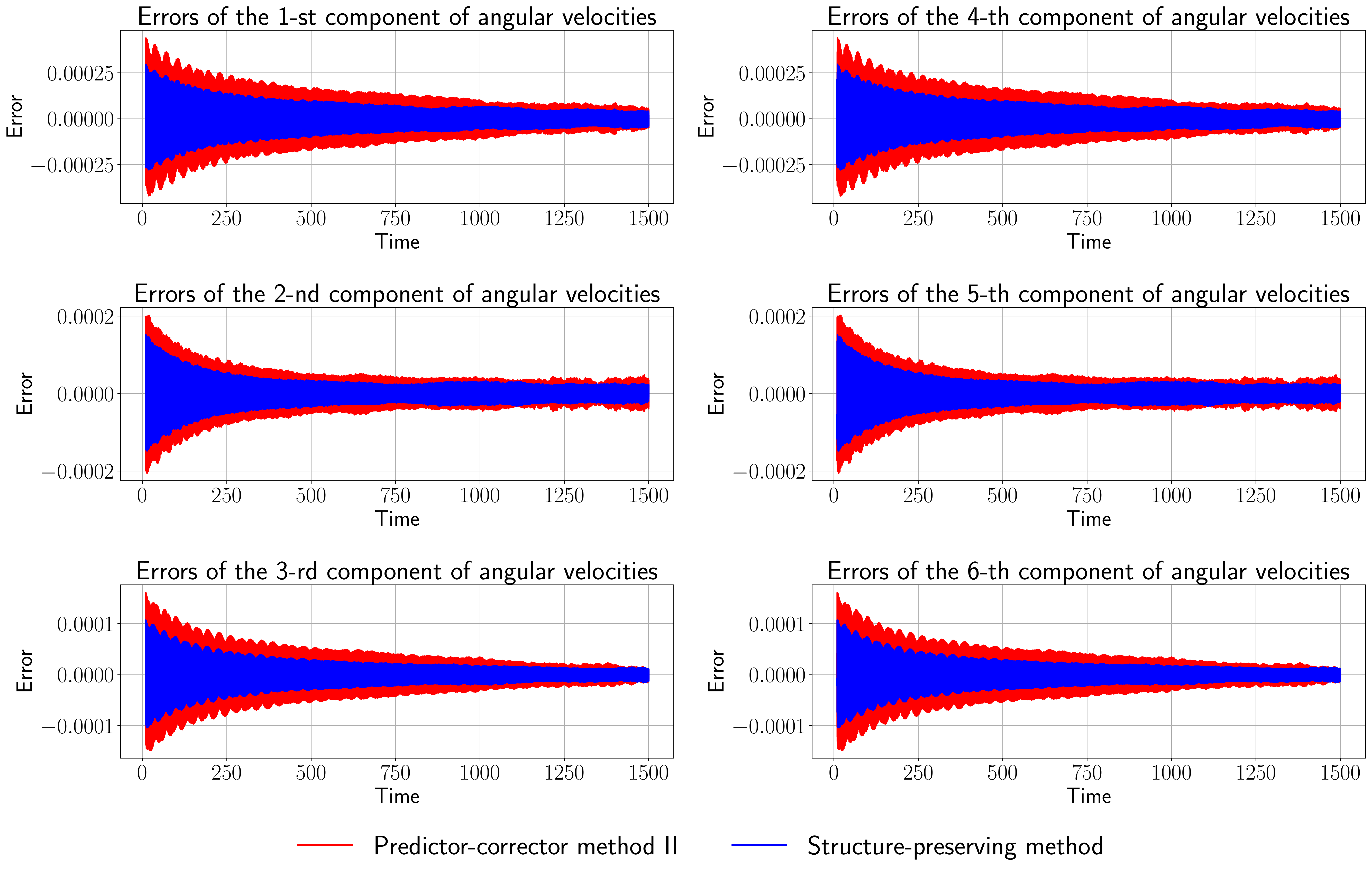}
 \caption{Simulation errors of angular velocities for $10\sim 1500\ \mathrm{s}$. The errors have the same meaning as those in Figure \ref{fig:10s}.}
 \label{fig:1500s}
\end{figure}

\section{Conclusions}\label{section 6}
In this article, we have presented predictor-corrector methods and structure-preserving methods for a generator system based on the first benchmark model of subsynchronous resonance. The structure-preserving property of the collocation methods has been illustrated by Dirac structure associated with port-Hamiltonian descriptor systems, which explains their advantage over PSCAD/EMTDC and predictor-corrector methods in terms of computational stability. Moreover, global error results of Gauss methods have guaranteed their effectiveness in numerical simulation. With appropriate initial conditions, these methods showed promising results in numerical simulations. 

\appendix
\section{Proof of Theorem \ref{Thm. global error of Gauss method}}\label{appendix}
First of all, we introduce the following convergence results for Runge-Kutta methods applied to index-1 differential-algebraic equations. 
\begin{theorem}\label{Thm. convergence result}
For an index 1 differential-algebraic system in the form of 
\begin{align}\label{Eq. differential-algebraic}
\left\{\begin{aligned}
&\dot{\widetilde{x}}=\widetilde{f}(t,\widetilde{x},\widetilde{y}), \\
&0=\widetilde{g}(t,\widetilde{x},\widetilde{y}), \\
\end{aligned}\right.     
\end{align}
assume that the initial values are consistent. Consider the Runge-Kutta method of classical order $p$, who satisfies
$$C(q):\ \sum\limits_{j=1}^s a_{ij}\gamma_j^{k-1}=\frac{\gamma_i^k}{k} \ \ \  i=1,\cdots,s,\ k=1,\cdots,q$$
with $p\geq q+1$ and has an invertible coefficient matrix $A=(a_{ij})$. Let $\rho=R(\infty)=1-b^{\top}A^{-1}e_s$, where $b=(b_1,\cdots,b_s)^{\top}$, $e_s=(1,\cdots,1)^{\top}\in \mathbb{R}^s$ and $R(w)=1+wb^{\top}(I-wA)^{-1}e_s$ is the stability function of this Runge-Kutta method. 

(1) If $b_i=a_{si}$ for all $i$, then the global error satisfies 
$$\widetilde{x}_n-\widetilde{x}(t_n)=\mathcal{O}(h^{p}),\ \widetilde{y}_n-\widetilde{y}(t_n)=\mathcal{O}(h^p)$$ for $t_n=nh\leq T_m$.

(2) If $-1\leq \rho <1$, then $$\widetilde{x}_n-\widetilde{x}(t_n)=\mathcal{O}(h^{p}),\ \widetilde{y}_n-\widetilde{y}(t_n)=\mathcal{O}(h^{q+1}).$$

(3) If $\rho=1$, then $$\widetilde{x}_n-\widetilde{x}(t_n)=\mathcal{O}(h^{p}),\ \widetilde{y}_n-\widetilde{y}(t_n)=\mathcal{O}(h^q).$$

(4) If $|\rho|>1$, then the numerical solution diverges.
\end{theorem}
\begin{proof}
See \cite[Theorem 3.1]{NSDAS}.
\end{proof}

It is evident that the constant $\rho$ plays a decisive role in Theorem \ref{Thm. convergence result}. For Gauss methods, we have the conclusion as follows. 

\begin{lemma}\label{Lem. rho}
For the $s$-stage Gauss method, the constant $\rho$ defined in Theorem \ref{Thm. convergence result} satisfies $\rho=(-1)^s$.
\end{lemma}

This conclusion has been presented in \cite[pp. 227]{SODE2} without proof, and can be proven through the property of Pad\'{e} approximation together with the fact that the stability function of $s$-stage Gauss method is the $(s,s)$-Pad\'{e} approximation (see \cite{HASM}). For the sake of completeness, we propose a straightforward approach to calculating $\rho$ for the $s$-stage Gauss method, which has no concern with the stability function. For this purpose, some useful lemmas will be presented.

Let 
\begin{align}\label{Eq. P_k}
P_k(\lambda)=\frac{\sqrt{2k+1}}{k!}\frac{\mathrm{d}^k}{\mathrm{d}\lambda^k}  \bigg(\lambda^k(\lambda-1)^k\bigg)=\sqrt{2k+1}\sum\limits_{m=0}^k (-1)^{m+k}\left(\begin{array}{c} k \\ m \end{array}\right)\left(\begin{array}{c} m+k \\ m \end{array}\right)\lambda^m  
\end{align}
be the shifted Legendre polynomials normalized such that 
\begin{align*}
\int_0^1 P_k^2(\lambda)d\lambda=1.    
\end{align*}
Then these polynomials satisfy the integration formulas
\begin{align}\label{Eq. P_k formula}
\int_0^{\lambda} P_0(\eta)d\eta=\xi_1 P_1(\lambda)+\frac{1}{2}P_0(\lambda), \
\int_0^{\lambda} P_k(\eta)d\eta=\xi_{k+1} P_{k+1}(\lambda)-\xi_k P_{k-1}(\lambda),
\end{align}
with $\xi_k=\frac{1}{2\sqrt{4k^2-1}}$ for $k=1,2,\cdots$. The first lemma is given as below. 
\begin{lemma}\label{Lem. G}
Suppose that $\gamma_1,\cdots,\gamma_s$ are the zero points of $s$th shifted Legendre polynomial and $B=\mathrm{diag}(b_1,\cdots,b_s)$, where $b_1,\cdots,b_s$ are the coefficients of $s$-stage Gauss method. Then the matrix 
\begin{align}\label{Eq. G}
G=\bigg(P_{j-1}(\gamma_i)\bigg)_{i,j=1,\cdots,s}    
\end{align}
satisfies $G^{\top}BG=I$.
\end{lemma}

\begin{proof} 
Similarly to the proof of \cite[Lemma 5.9]{SODE2}, the polynomials $P_k(\lambda)P_l(\lambda)$ ($k+l\leq 2s-2$) can be exactly integrated by Gauss quadrature formula, which means 
\begin{align*}
\sum\limits_{i=1}^s b_i P_k(\gamma_i)P_l(\gamma_i)=\int_0^1 P_k(\lambda)P_l(\lambda)d\lambda=\delta_{kl}.       
\end{align*}
This implies that $G^{\top}BG=I$.
\end{proof}

Lemma \ref{Lem. G} indicates that the matrix $G$ is nonsingular, thus we have the following lemma.
\begin{lemma}\label{Lem. X_G}
Let $A=(a_{ij})$ be the coefficient matrix for the $s$-stage Gauss method, then it can be obtained that 
\begin{align}\label{Eq. X_G}
G^{-1}AG=\left(\begin{array}{ccccc}
    1/2 & -\xi_1 &  &  &  \\
    \xi_1 & 0 & -\xi_2 &  &  \\
       & \xi_2 & \ddots & \ddots &  \\
       &  & \ddots & 0 & -\xi_{s-1} \\
       &  &  & \xi_{s-1} & 0
\end{array}\right)=:X_G.  
\end{align}
\end{lemma}

\begin{proof}
See details in \cite[Theorem 5.6]{SODE2}. In brief, $C(q)$ means the quadrature formulas with nodes $\gamma_1,\cdots,\gamma_s$ and weights $a_{i1},\cdots,a_{is}$ which can exactly integrate polynomials up to degree $q-1$ on the interval $[0,\gamma_i]$ ($i=1,\cdots,s$). Combining this conclusion with (\ref{Eq. P_k formula}) results in 
\begin{align*}
&\sum\limits_{j=1}^s a_{ij}P_0(\gamma_j)=\int_0^{\gamma_i} P_0(\lambda)d\lambda=\xi_1 P_1(\gamma_i)+\frac{1}{2}P_0(\gamma_i), \\
&\sum\limits_{j=1}^s a_{ij}P_k(\gamma_j)=\int_0^{\gamma_i} P_k(\lambda)d\lambda=\xi_{k+1} P_{k+1}(\gamma_i)-\xi_k P_{k-1}(\gamma_i) \ \ k=1,\cdots,q-1.
\end{align*}
Then insert (\ref{Eq. G}) into the equations above and write them in matrix form, afterwards (\ref{Eq. X_G}) can be obtained by the fact that $P_s(\gamma_1)=\cdots=P_s(\gamma_s)=0$ and $G$ is nonsingular.
\end{proof}

Lemma \ref{Lem. X_G} has introduced a tridiagonal matrix $X_G$ which is similar to $A$, hence the determinant of $A$ can be calculated through the relation $\mathrm{det}A=\mathrm{det}X_G$.
\begin{lemma}\label{Lem. detA}
$\mathrm{det}A=\mathrm{det}X_G=\frac{s!}{(2s)!}$.    
\end{lemma}

\begin{proof}
Take
\begin{align*}
D_k=\left|\begin{array}{cccc}
       0 & -\xi_k &  &  \\
       \xi_k & \ddots & \ddots &  \\
         & \ddots & 0 & -\xi_{s-1} \\
         &  & \xi_{s-1} & 0
\end{array}\right|,    
\end{align*}
then we have 
\begin{align*}
\begin{split}
D_k&=\left|\begin{array}{cccccc}
       0 & -\xi_k &  &  &  & \\
       \xi_k & 0 & -\xi_{k+1} &  &  &  \\
         & \xi_{k+1} & 0 & \ddots & \\
         &  & \ddots & \ddots & \ddots & \\
         &  &  & \ddots & 0 & -\xi_{s-1} \\
         &  &  &  & \xi_{s-1} & 0
\end{array}\right| \\
&=\xi_k \left|\begin{array}{cccccc}
       \xi_k & -\xi_{k+1} &  &  &  & \\
       0 & 0 & -\xi_{k+2} &  &  &  \\
         & \xi_{k+2} & 0 & \ddots & \\
         &  & \ddots & \ddots & \ddots & \\
         &  &  & \ddots & 0 & -\xi_{s-1} \\
         &  &  &  & \xi_{s-1} & 0
\end{array}\right|=\xi_k^2 \left|\begin{array}{cccc}
       0 & -\xi_{k+2} &  &  \\
       \xi_{k+2} & \ddots & \ddots &  \\
         & \ddots & 0 & -\xi_{s-1} \\
         &  & \xi_{s-1} & 0
\end{array}\right| \\
&=\xi_k^2 D_{k+2}
\end{split}
\end{align*}
for $k\leq s-3$. Repeated insertion of this formula contributes to 
\begin{align}\label{Eq. D_2 for even s}
\begin{split}
&D_2=\xi_2^2 \xi_4^2 \cdots \xi_{2m-4}^2 D_{2m-2}=\prod\limits_{n=1}^{m-2}\xi_{2n}^2 \left|\begin{array}{ccc}
       0 & -\xi_{2m-2} & 0 \\
       \xi_{2m-2} & 0 & -\xi_{2m-1} \\
       0 & \xi_{2m-1} & 0 \\
\end{array}\right|=0, \\
&D_3=\xi_3^2 \xi_5^2 \cdots \xi_{2m-3}^2 D_{2m-1}=\prod\limits_{n=1}^{m-2}\xi_{2n+1}^2 \left|\begin{array}{cc}
       0 & -\xi_{2m-1} \\
       \xi_{2m-1} & 0 \\
\end{array}\right|=\prod\limits_{n=1}^{m-1}\xi_{2n+1}^2 
\end{split}
\end{align}
when $s=2m\ (m\in \mathbb{N}^+)$. In the case that $s=2m-1\ (m\in \mathbb{N}^+)$, (\ref{Eq. D_2 for even s}) becomes
\begin{align}\label{Eq. D_2 for odd s}
D_2=\prod\limits_{n=1}^{m-1}\xi_{2n}^2,\ D_3=0.  
\end{align}
Consider that $\xi_k=\frac{1}{2\sqrt{4k^2-1}}\ (k=1,2,\cdots)$, it can be deduced that for $s=2m$,
\begin{align}\label{Eq. detA even}
\begin{split}
\mathrm{det}X_G&=\frac{1}{2}D_2+\xi_1^2 D_3=\prod\limits_{n=0}^{m-1}\xi_{2n+1}^2=\prod\limits_{n=0}^{m-1}\frac{1}{4(4n+1)(4n+3)} \\
&=\prod\limits_{n=0}^{m-1}\frac{(2n+1)(2n+2)}{(4n+1)(4n+2)(4n+3)(4n+4)}=\frac{(2m)!}{(4m)!}=\frac{s!}{(2s)!}.
\end{split}  
\end{align}
On the other hand, for $s=2m-1$ we have 
\begin{align}\label{Eq. detA odd}
\begin{split}
\mathrm{det}X_G&=\frac{1}{2}D_2+\xi_1^2 D_3=\frac{1}{2}\prod\limits_{n=1}^{m-1}\xi_{2n}^2=\frac{1}{2}\prod\limits_{n=1}^{m-1}\frac{1}{4(4n-1)(4n+1)} \\
&=\frac{1}{1\cdot2}\prod\limits_{n=1}^{m-1}\frac{2n(2n+1)}{(4n-1)4n(4n+1)(4n+2)}=\frac{(2m-1)!}{(4m-2)!}=\frac{s!}{(2s)!}.
\end{split}  
\end{align}
Therefore, we reach the conclusion of this lemma by (\ref{Eq. detA even}) and (\ref{Eq. detA odd}).
\end{proof}

The nonsingularity of matrix $A$ can be immediately verified by Lemma \ref{Lem. detA}, then Lemma \ref{Lem. rho} will be proven based on Lemma \ref{Lem. G}$\sim$\ref{Lem. detA}.

\begin{proof}[Proof of Lemma \ref{Lem. rho}]
On the basis of Lemma \ref{Lem. G}$\sim$\ref{Lem. X_G}, we immediately attain
\begin{align}\label{Eq. calculation of rho 1}
b^{\top}A^{-1}e_s=e_s^{\top}B^{\top}\left(G X_G^{-1} G^{-1}\right)e_s=\left(G^{\top}B e_s\right)^{\top} X_G^{-1} \left(G^{\top}B e_s\right).   
\end{align}
Consider that 
\begin{align*}
G^{\top}B e_s&=\left(\begin{array}{cccc}
    P_0(\gamma_1) & P_0(\gamma_2) & \cdots & P_0(\gamma_s) \\
    P_1(\gamma_1) & P_1(\gamma_2) & \cdots & P_1(\gamma_s) \\
    \vdots & \vdots &  &  \vdots \\
    P_{s-1}(\gamma_1) & P_{s-1}(\gamma_2) & \cdots & P_{s-1}(\gamma_s)
\end{array}\right)\left(\begin{array}{c} b_1 \\ b_2 \\ \vdots \\ b_s \end{array}\right) \\
&=\left(\begin{array}{c}
    \sum\limits_{i=1}^s b_i P_0(\gamma_i)  \\
    \sum\limits_{i=1}^s b_i P_1(\gamma_i)  \\
    \vdots \\
    \sum\limits_{i=1}^s b_i P_{s-1}(\gamma_i)
\end{array}\right)=\left(\begin{array}{c}
    \int_0^1 P_0(\lambda)d\lambda  \\
    \int_0^1 P_1(\lambda)d\lambda  \\
    \vdots \\
    \int_0^1 P_{s-1}(\lambda)d\lambda
\end{array}\right) 
\end{align*}
and
\begin{align*}
\int_0^1 P_0(\lambda)d\lambda &=\int_0^1 d\lambda=1,\\
\int_0^1 P_k(\lambda)d\lambda &=\frac{\sqrt{2k+1}}{k!}\frac{\mathrm{d}^{k-1}}{\mathrm{d}\lambda^{k-1}}  \bigg(\lambda^k(\lambda-1)^k\bigg)\Bigg|_0^1 \\
&=\frac{\sqrt{2k+1}}{k!}\sum\limits_{n=0}^{k-1}\left(\begin{array}{c}
    k-1 \\ n
\end{array}\right)\frac{k!}{(k-n)!}\lambda^{k-n}\frac{k!}{(n+1)!}(\lambda-1)^{n+1}\Bigg|_0^1=0,\ k\geq 1,
\end{align*}
then (\ref{Eq. calculation of rho 1}) together with (\ref{Eq. D_2 for even s}), (\ref{Eq. D_2 for odd s}) implies that
\begin{align}\label{Eq. calculation of rho 2}
b^{\top}A^{-1}e_s=\left(X_G^{-1}\right)_{11}=\frac{1}{\mathrm{det}X_G}D_2=\left\{\begin{array}{cc}
   0, & \hbox{for even}\ s,  \\
   2, & \hbox{for odd}\ s.
\end{array}\right.   
\end{align}
Here $\left(X_G^{-1}\right)_{11}$ represents the $(1,1)$-element of the matrix $X_G^{-1}$, which can be calculated by 
$$X_G^{-1}=\frac{1}{\mathrm{det}X_G}X_G^{\ast}$$
with the adjoint matrix $X_G^{\ast}$. Finally, (\ref{Eq. calculation of rho 2}) leads to $\rho=(-1)^s$, which completes the proof.
\end{proof}

On the foundation of previous conclusions, it is time to prove Theorem \ref{Thm. global error of Gauss method}.

\begin{proof}[Proof of Theorem \ref{Thm. global error of Gauss method}]
Take $\widetilde{x}=\left(\widetilde{\boldsymbol{\Psi}}; \dot{\boldsymbol{\theta}}; \boldsymbol{\theta}\right),\ \widetilde{y}=\dot{\widetilde{\boldsymbol{\Psi}}}$, then $x(t)$ given by (\ref{Eq. state, input, output}) can be written as $x(t)=\left(\widetilde{y}(t);\widetilde{x}(t);t\right)$. Let $x_0=\left(\widetilde{y}_n;\widetilde{x}_n;t_n\right)$ be the numerical solution to (\ref{Eq. collocation for system}) at $t_n=nh$, and $x_f=\left(\widetilde{y}_{n+1};\widetilde{x}_{n+1};t_{n+1}\right)$ be the numerical solution at $t_{n+1}=(n+1)h$. Therefore, if we set $k_i=\left(\dot{\widetilde{Y}}_{ni};\dot{\widetilde{X}}_{ni};1\right)$, then (\ref{Eq. collocation for system}) is equivalent to  
\begin{subequations}\label{Eq. direct approach}
\begin{align}
&\widetilde{x}_{n+1}=\widetilde{x}_n+h\sum\limits_{i=1}^s b_i \dot{\widetilde{X}}_{ni},\ \widetilde{X}_{ni}=\widetilde{x}_n+h\sum\limits_{j=1}^s a_{ij} \dot{\widetilde{X}}_{nj}, \\
&\widetilde{y}_{n+1}=\widetilde{y}_n+h\sum\limits_{i=1}^s b_i \dot{\widetilde{Y}}_{ni},\ \widetilde{Y}_{ni}=\widetilde{y}_n+h\sum\limits_{j=1}^s a_{ij} \dot{\widetilde{Y}}_{nj}, \\
&\dot{\widetilde{X}}_{ni}=\widetilde{f}\left(t_n+\gamma_i h,\widetilde{X}_{ni},\widetilde{Y}_{ni}\right),\
0=\widetilde{g}\left(t_n+\gamma_i h,\widetilde{X}_{ni},\widetilde{Y}_{ni}\right),
\end{align}
\end{subequations}
where 
\begin{align}\label{Eq. new f,g}
\begin{split}
 &\widetilde{f}(t,\widetilde{x},\widetilde{y})=\left(\begin{array}{c}
\dot{\widetilde{\boldsymbol{\Psi}}} \\
-J^{-1}D\dot{\boldsymbol{\theta}}-J^{-1}K\boldsymbol{\theta}+J^{-1}\left(T-\frac{1}{2}\widetilde{\boldsymbol{\Psi}}^{\top} \frac{\partial \widetilde{\Gamma}(\boldsymbol{\theta})}{\partial \boldsymbol{\theta}} \widetilde{\boldsymbol{\Psi}}\right) \\
\dot{\boldsymbol{\theta}} \\
\end{array}\right),  \\
&\widetilde{g}(t,\widetilde{x},\widetilde{y})=\widetilde{K}_R\dot{\widetilde{\boldsymbol{\Psi}}}+\left(\widetilde{K}_L+\widetilde{\Gamma}(\boldsymbol{\theta})\right) \widetilde{\boldsymbol{\Psi}}-\widetilde{I}_s(t).    
\end{split}
\end{align}
In reality, (\ref{Eq. direct approach}) coincides with the collocation method, i.e. $s$-stage Runge-Kutta method with the coefficients given by (\ref{Eq. collocation coefficient}), applied to the generator system in (\ref{Eq. differential-algebraic}) form through the \textit{direct approach}, whose details can be seen in \cite[pp. 24$\sim$25]{NSDAS} and \cite[Section 5.2]{DAE}. Since the Jacobian matrix $\widetilde{g}_{\widetilde{y}}(t,\widetilde{x},\widetilde{y})=\widetilde{K}_R$ is always nonsingular with a bounded inverse (see (\ref{Eq. new coefficient})), the generator system in (\ref{Eq. differential-algebraic}) form is an index-1 differential-algebraic system according to \cite[(1.4)$\sim$(1.5)]{NSDAS}. Notice that the $s$-stage Gauss method is of order $2s$ with $C(s)$ satisfied, thus Theorem \ref{Thm. global error of Gauss method} can be proven through Theorem \ref{Thm. convergence result} together with the nonsingularity of matrix $A$ and Lemma \ref{Lem. rho}.
\end{proof}

\end{document}